\documentclass{amsart}

\usepackage[utf8]{inputenc} 
\usepackage[T1]{fontenc}    
\usepackage{lmodern}    
\usepackage{hyperref}       
\usepackage{url}            
\usepackage{booktabs}       
\usepackage{amsfonts}       
\usepackage{nicefrac}       
\usepackage{microtype}      
\usepackage{amssymb}
\usepackage{constants}

\usepackage[numbers]{natbib}

\hypersetup{breaklinks=true,
    linkcolor=blue,
    citecolor=red,
    colorlinks=true,
pdfborder={0 0 0}}

\usepackage{amsmath}
\usepackage{amsmath}
\usepackage{amsthm}
\usepackage{amsfonts}
\usepackage{amssymb}

\usepackage{hyperref}
\usepackage{cleveref}

\usepackage{xcolor}
\usepackage{color,soul}

\def\bin{\textsc{Bin}}
\def\eps{\varepsilon}

\def\argmin{\mathop{\rm arg\, min}}
\def\real{{\mathbb{R}}}
\def\R{{\real}}
\def\E{{\mathbb E}}

\def\simplex{\Lambda_M}

\def\mydefb#1{\expandafter\def\csname b#1\endcsname{{\boldsymbol{#1}}}}
\def\mydefallb#1{\ifx#1\mydefallb\else\mydefb#1\expandafter\mydefallb\fi}
\mydefallb ABCDEFGHIJKLMNOPQRSTUVWXYZ\mydefallb
\mydefallb abcdeghijklmnopqrstuvwxyz\mydefallb
\DeclareMathOperator{\trace}{trace}

\title[Optimally tuning Tikhonov regularizers]{The cost-free nature of optimally tuning Tikhonov regularizers
and other ordered smoothers}

\author[P.C. Bellec]{Pierre C. Bellec$^*$}
\thanks{$^*$: 
Department of Statistics, Busch Campus,
Rutgers University, Piscataway, NJ 08854, USA.
}
\thanks{$^*$:
Research partially supported by the NSF Grant DMS-1811976.
}

\author[D. Yang]{Dana Yang$^\dagger$}
\thanks{$^\dagger$: 
Department of Statistics \& Data Science, 
Yale University, New Haven, CT 06511, USA.
}

\usepackage{thmtools}

\declaretheorem[name=Theorem,numberwithin=section]{theorem}

\declaretheorem[name=Lemma,sibling=theorem]{lemma}
\declaretheorem[name=Corollary,sibling=theorem]{corollary}

\declaretheorem[name=Definition,style=definition]{definition}

\def\df{{\mathop{{\rm df}}}{}}

\numberwithin{equation}{section}

\begin{document}

\maketitle

\begin{abstract}
  We consider the problem of selecting the best estimator among
  a family of Tikhonov regularized estimators, or, alternatively, to select
  a linear combination of these regularizers that is as good as the
  best regularizer in the family.
  Our theory reveals that if the Tikhonov regularizers share the same penalty matrix
  with different tuning parameters,
  a convex procedure based on $Q$-aggregation achieves the mean square
  error of the best estimator, up to a small error term no larger than $C\sigma^2$, where $\sigma^2$ is the noise level and $C>0$ is an absolute constant.
  Remarkably, the error term does not depend on the penalty matrix
  or the number of estimators
  as long as they share the same penalty matrix, i.e., it applies to any grid of tuning parameters,
  no matter how large the cardinality of the grid is.
  This reveals the surprising "cost-free" nature of optimally tuning Tikhonov regularizers,
  in striking contrast with the existing literature on aggregation
  of estimators where one typically has to pay a cost of $\sigma^2\log(M)$ where
  $M$ is the number of estimators in the family.
  The result holds, more generally, for any family of ordered linear smoothers.
  This encompasses Ridge regression as well as Principal Component Regression.
  The result is extended to the problem of tuning Tikhonov regularizers
  with different penalty matrices.
\end{abstract}

\section{Introduction}

Consider a learning problem where one is given an observation vector $y\in\R^n$ and
a design matrix $X\in\R^{n\times p}$. Given a positive definite matrix $K\in\R^{p\times p}$
and a regularization parameter $\lambda > 0$, the Tikhonov regularized estimator
$\hat w(K,\lambda)$ is defined as the solution of the quadratic program
\begin{equation}
    \textstyle
    \hat w(K,\lambda) = \argmin_{w \in \R^p}
    \left(
    \|X w - y\|^2 + \lambda w^T K w
    \right),
\end{equation}
where $\|\cdot\|$ is the Euclidean norm.
Since we assume that the penalty matrix
$K$ is positive definite, the above optimization problem
is strongly convex and the solution is unique.
In the special case $K=I_{p\times p}$, the above estimator reduces to Ridge regression.
It is well known that the above optimization problem can be explicitly solved
and that
\begin{align*}
    \hat w(K, \lambda) 
    &=(X^TX + \lambda K)^{-1}X^T y \\
    &=K^{-1/2} (K^{-1/2} X^T X K^{-1/2} + \lambda I_{p\times p} )^{-1} K^{-1/2} X^T
    .
\end{align*}

\subsection*{Problem statement.}
Consider the Gaussian mean model
\begin{equation}
    y=\mu+\eps \qquad \text{ with }\qquad \eps\sim N(0,\sigma^2 I_{n\times n})
    \label{model}
\end{equation}
where $\mu\in\R^n$ is an unknown mean, and consider a deterministic design matrix $X\in\R^{n\times p}$. 
We are given a grid of tuning parameters $\lambda_1,...,\lambda_M\ge 0$ and a penalty matrix $K$
as above. Our goal is to construct an estimator $\tilde w$ such that the \emph{regret}
or \emph{excess risk}
\begin{equation}
    \E[ \|X \tilde w - \mu\|^2 ] - \min_{j=1,...,M}\E[\|X\hat w (K,\lambda_j)-\mu\|^2]
    \label{regret-tikhonov}
\end{equation}
is small.
Beyond the construction of an estimator $\tilde w$ that has small regret,
we aim to answer the following questions:
\begin{itemize}
    \item How does the worst-case regret scales with $M$, the number of tuning
        parameters on the grid?
    \item How does the worst case regret scales with $R^*=\min_{j=1,...,M} \E[ \|X\hat w
        (K,\lambda_j)-\mu\|^2]$, the minimal mean squared error among the tuning parameters $\lambda_1,...,\lambda_M$?
\end{itemize}

\subsection*{Ordered linear smoothers.} If $A_j=X(X^TX + \lambda_j K)X^T$ is the matrix such that
$A_j y = X \hat w(K,\lambda_j)$, the family of estimators $\{A_j, j=1,...,M\}$ is an example
of ordered linear smoothers, introduced \cite{kneip1994ordered}.

\begin{definition}
    \label{def:ordered-linear-smoothers}
The family of $n\times n$ matrices $\{A_1,...,A_M\}$ are referred to
as ordered linear smoothers if (i) $A_j$ is symmetric and $0\le w^T A_j w \le \|w\|^2$ for all $w\in\R^p$ and all $j=1,...,M$, (ii) the matrices commute: $A_jA_k = A_kA_j$ for all $j,k=1,...,M$, 
and (iii) either $A_j \preceq A_k$ or $A_k \preceq A_j$ holds for all $j,k=1,...,M$,
where $\preceq$ denotes the partial order of positive symmetric matrices, i.e.,
$A\preceq B$ if and only if $B-A$ is positive semi-definite.
\end{definition}

Condition (i) is mild: if the matrix $A$ is not symmetric then it is not admissible
and there exists a symmetric matrix $A'$ such that $\E[\|A'y - \mu\|^2] \le \E[\|Ay - \mu\|^2]$ with a strict inequality for at least one $\mu\in\R^n$ \cite{cohen1966all}, so we may as well replace $A$ with the symmetric matrix $A'$.
Similarly, if $A$ is symmetric
with some eigenvalues outside of $[0,1]$, then $A$ is not admissible and
there exists another symmetric matrix $A'$ with eigenvalues in $[0,1]$ and
smaller prediction error for all $\mu\in\R^n$, and strictly smaller prediction error for at least one $\mu\in\R^n$ if $n\ge 3$
\cite{cohen1966all}.

Conditions (ii) and (iii) are more stringent: they require that the matrices can be
diagonalized in the same orthogonal basis $(u_1,...,u_k)$ of $\R^n$, and that the matrices
are ordered in the sense that there exists $n$ functions $\alpha_1,...,\alpha_n:\R\to[0,1]$, either all non-increasing or all non-decreasing, such that
\begin{equation}
\{A_1,...,A_M\} \subset \{\alpha_1(\lambda) u_1 u_1^T + ... + \alpha_n(\lambda) u_n u_n^T, \lambda\in\R\},
\label{eq:linear-smoothers}
\end{equation}
see \cite{kneip1994ordered} for a rigorous proof of this fact.
A special case of particular interest is the above Tikhonov regularized estimators,
which satisfies conditions (i)-(ii)-(iii). In this case, the matrix
$A_j= X (X^TX + \lambda_j K)^{-1} X^T$ is such that $A_j y = X \hat w(K,\lambda_j)$.
To see that for any grid of tuning parameters $\lambda_1,...,\lambda_M$,
the Tikhonov regularizers form a family of ordered linear smoothers,
the matrix $A_j$ can be rewritten as $A_j = B(B^TB + \lambda_j I_{p\times p})^{-1} B^T$
where $B$ is the matrix $X K^{-1/2}$. From this expression of $A_j$, it is clear
that $A_j$ is symmetric, that $A_j$ can be diagonalized in the orthogonal basis
made of the left singular vectors of $B$, and that the eigenvalues of $A_j$ are decreasing 
functions of the tuning parameter. Namely, the $i$-th eigenvalue of $A_j$ is equal
to $\alpha_i(\lambda_j) = \mu_i(B)^2/(\mu_i(B)^2 + \lambda_j)$ where $\mu_i(B)$ is the $i$-th singular value of $B$.

\subsection*{Overview of the literature.}
There is a substantial amount of literature related to this problem, starting
with \cite{kneip1994ordered} where ordered linear smoothers are introduced
and where their properties were first studied.
\citet{kneip1994ordered} proves that if $A_1,...,A_M$ are ordered linear smoothers,
then selecting the estimate with the smallest $C_p$ criterion \cite{mallows1973some}, i.e.,
\begin{equation}
    \textstyle
    \hat k = \argmin_{j=1,...,M} C_p(A_j), \quad\text{where}\quad C_p(A) = \|A y - y\|^2 + 2 \sigma^2 \trace(A_j),
\end{equation}
leads to the regret bound (sometimes referred to as \emph{oracle inequality})
\begin{equation}
    \E[\|A_{\hat k} y - \mu\|^2] - R^*
    \le C \sigma \sqrt{R^*} + C\sigma^2,
    \quad\text{ where }\quad
    R^* = \min_{j=1,...,M} \E[\|A_j y - \mu\|^2]
    \label{oracle-inequality-kneip}
\end{equation}
for some absolute constant $C>0$.
This result was later improved
in
\cite[Theorem 3]{golubev2010universal}
\cite{chernousova2013ordered}
using an estimate based on exponential weighting,
showing that the regret is bounded from above by $\sigma^2\log(2+R^*/\sigma^2)$.

Another line of research has obtained regret bounds that scales with the cardinality
$M$ of the given family of linear estimators.
Using an exponential weight estimate with a well chosen temperature parameter,
\cite{leung2006information,dalalyan2012sharp} showed that if $A_1,...,A_M$ are squared matrices of size $n$ that are either orthogonal projections, or that satisfies some commutativity property, then a data-driven convex combination $\hat A_{EW}$ of the matrices
$A_1,...,A_M$ satisfies
\begin{equation}
    \E[\|\hat A_{EW} y - \mu\|^2] - R^*
    \le C \sigma^2 \log M.
    \label{oracle-inequality-EW}
\end{equation}
where $C>0$ is an absolute constant.
This was later improved in \cite{bellec2014affine} using an estimate from
the $Q$-aggregation procedure of \cite{dai2012deviation,dai2014aggregation}. Namely, Theorem 2.1 in \cite{bellec2014affine} states that if $A_1,...,A_M$ are squared matrices with operator norm at most 1,
then
\begin{equation}
    \mathbb P\Big(
    \|\hat A_Q y - \mu\|^2] - \min_{j=1,...,M}\|A_j y - \mu\|^2
    \le C \sigma^2 \log(M/\delta)
    \Big) \ge 1-\delta
    \label{oracle-inequality-Q}
\end{equation}
for any $\delta\in(0,1)$, where $\hat A_Q$ is a data-driven convex combination
of the matrices $A_1,...,A_M$. A result similar to
\eqref{oracle-inequality-EW} can then be deduced from the above high probability
bound by integration.  It should be noted that the linear estimators in
\eqref{oracle-inequality-EW} and \eqref{oracle-inequality-Q} need not be
ordered smoothers (the only assumption in in \eqref{oracle-inequality-Q} is
that the operator norm of $A_j$ is at most one), unlike
\eqref{oracle-inequality-kneip} where the ordered smoothers assumption is key.

Another popular approach to select a good estimate among a family of
linear estimators is the Generalized Cross-Validation (GCV) criterion of
\cite{craven1978smoothing,golub1979generalized}.
If we are given $M$ linear estimators defined by square matrices $A_1,...,A_M$,
Generalized Cross-Validation selects the estimator
$$\hat k = \argmin_{j=1,...,M}\left(\|A_j y - y\|^2 /(\trace[I_{n\times n}-A_j] )^2 \right).$$
We could not pinpoint in the literature an oracle inequality satisfied by GCV comparable to \eqref{oracle-inequality-kneip}-\eqref{oracle-inequality-EW}-\eqref{oracle-inequality-Q}, though we mention that
\cite{li1986asymptotic} exhibits asymptotic frameworks where GCV is 
suboptimal while, in the same asymptotic frameworks, Mallows $C_p$ is optimal.

The problem of optimally tuning Tikhonov regularizers, Ridge regressors
or smoothning splines has received considerable attention in the last four decades
(for instance, the GCV paper \cite{golub1979generalized} is cited more than four thousand times) and the authors of the present paper are guilty of numerous omissions
of important related works.
We refer the reader to the recent surveys \cite{arlot2010survey,arlot2009data}
and the references therein for the problem of tuning linear
estimators, and to \cite{tsybakovICM} for a survey of aggregation results.

Coming back to our initial problem of optimally tuning a family of Tikhonov
regularizers $\hat w(K,\lambda_1),...,\hat w(K,\lambda_M)$,
the results \eqref{oracle-inequality-kneip}, \eqref{oracle-inequality-EW}
and \eqref{oracle-inequality-Q} above suggest that one must pay a price
that depends either on the cardinality $M$ of the grid of tuning parameters,
or on $R^*=\min_{j=1,...,M}\E[\|X\hat w(K,\lambda_j) - \mu\|^2]$,
the minimal mean squared error on this grid.

\subsection*{
Optimally tuning ordered linear smoothers incurs no statistical cost.} Surprisingly, our theoretical results of the next
sections reveal that if $A_1,...,A_M$ are ordered linear smoothers,
for example Tikhonov regularizers sharing the same penalty matrix $K$,
then it is possible to construct a data-driven convex combination $\hat A$ of $A_1,...,A_M$ such
that the regret satisfies
$$\E[ \|\hat A y - \mu\|^2] - \min_{j=1,...,M}\E[\|A_j y - \mu\|^2]
\le \Cl{punchline}\sigma^2$$
for some absolute constant $\Cr{punchline}>0$.
Hence the regret in \eqref{regret-tikhonov} is bounded by $\Cr{punchline}\sigma^2$,
an upper bound that is (a) independent of the cardinality $M$ of the grid of tuning parameters
and (b) independent of the minimal risk $R^*=\min_{j=1,...,M}\E[\|A_j y - \mu\|^2]$.
No matter how coarse the grid of tuning parameter is, no matter the number of tuning parameters to choose from, no matter how large the minimal risk $R^*$ is, the regret
of the procedure constructed in the next section is always bounded by $\Cr{punchline}\sigma^2$.

\paragraph{Notation.} Throughout the paper, $C_1,C_2,C_3...$ denote absolute positive constants. The norm $\|\cdot\|$ is the Euclidean norm of vectors. Let $\|\cdot\|_{op}$ and
$\|\cdot\|_F$ be the operator and Frobenius norm of matrices.

\section{Construction of the estimator}
\label{sec:construction-estimator}

Assume that we are given $M$ matrices $A_1,...,A_M$, each matrix corresponding
to the linear estimator $A_j y$. \citet{mallows1973some} $C_p$ criterion is given by
\begin{equation}
    \label{C_p}
    C_p(A) \triangleq \|A y - y\|^2  + 2\sigma^2 \trace A
\end{equation}
for any square matrix $A$ of size $n\times n$. Following several works on aggregation
of estimators \cite{nemirovski2000lecture,tsybakov2003optimal,leung2006information,rigollet2007linear,dalalyan2012sharp,dai2012deviation,bellec2014affine}
we parametrize the convex hull of the matrices $A_1,...,A_M$ as follows:
\begin{equation}
    \label{notation-A_theta}
    A_\theta \triangleq \sum_{j=1}^M \theta_j A_j, \quad
    \text{for each }\theta \in \simplex,
    \quad\text{ where }\quad \Lambda_M 
    = \Big\{ \theta\in\R^M : \theta_j\ge 0, \sum_{j=1}^M \theta_j = 1 \Big\}.
\end{equation}
Above, $\Lambda_M$ is the simplex in $\R^M$ and the convex hull of the matrices
$A_1,...,A_M$ is exactly the set $\{A_\theta, \theta\in \simplex \}$.
Finally, define the weights $\hat\theta\in\simplex$ by
\begin{equation}
    \label{optimization-problem}
    \hat\theta = \argmin_{\theta\in\simplex}
    \Big(
        C_p(A_\theta) + \frac 1 2 \sum_{j=1}^M \theta_j \|(A_\theta- A_j)y\|^2
    \Big).
\end{equation}
The first term of the objective function is Mallows $C_p$ from \eqref{C_p},
while the second term is a penalty derived from the $Q$-aggregation
procedure from \cite{rigollet2012kullback,dai2012deviation}.
The penalty is minimized at the vertices of the simplex and thus penalizes the interior of $\Lambda_M$.
Although convexity of the above optimization problem is unclear at first sight
because the penalty is non-convex,
the objective function can be rewritten, thanks to a bias-variance decomposition, as
\begin{equation}
\textstyle
    \frac 1 2 \|A_\theta y - y\|^2 + 2 \sigma^2\trace(A_\theta)+ \frac 1 2 \sum_{j=1}^M \theta_j \|A_jy - y\|^2.
    \label{eq:objective-function-clearly-convex}
\end{equation}
The first term is a convex quadratic form in $\theta$, while both the second
term $(2\sigma^2 \trace[A_\theta])$ and the last term are linear in $\theta$. 
It is now clear that the objective function is convex and \eqref{optimization-problem}
is a convex quadratic program (QP) with $M$ variables and $M+1$ linear constraints.
The computational complexity of such convex QP is polynomial and well studied,
 e.g., \cite[page 304]{Vavasis2001}.
The final estimator is 
\begin{equation}
    \textstyle
    \hat y \triangleq A_{\hat \theta} y = \sum_{j=1}^M \hat\theta_j A_j y,
\end{equation}
a weighted sum of the values predicted by the linear estimators $A_1,...,A_j$.
The performance of this procedure is studied in \cite{dai2014aggregation,bellec2014affine};
\cite{bellec2014affine} derived the oracle inequality \eqref{oracle-inequality-Q} which
is optimal for certain collections $\{A_1,...,A_m\}$.
However, we are not aware of previous analysis of this procedure in the context
of ordered linear smoothers.

\section{Constant regret for ordered linear smoothers}

\begin{theorem}
    \label{thm:main}
    The following holds for absolute constants
    $\Cr{punchline},\Cr{lem56},\Cr{series-probability}>0$.
    Consider the Gaussian mean model \eqref{model}.
    Let $\{A_1,...,A_M\}$ be a family ordered linear smoothers as in \Cref{def:ordered-linear-smoothers}.
    Let $\hat\theta$ be the solution to the optimization problem \eqref{optimization-problem}. Then
    $\hat y = A_{\hat \theta}y$ enjoys the regret bound
    \begin{equation}
        \E[ \|A_{\hat\theta}y - \mu\|^2] - \min_{j=1,...,M}\E[\|A_j y - \mu\|^2] \le \Cr{punchline}\sigma^2
        .
        \label{regret-bound-main-theorem-expectation}
    \end{equation}
    Furthermore, if $j_* =\argmin_{j=1,...,M}\E[\|A_j y - \mu\|^2]$ has minimal risk
    then for any $x\ge 1$,
    \begin{equation}
        \label{eq:main-thm:probability}
        \mathbb P\left\{
            \|A_{\hat\theta}y - \mu\|^2 - \|A_{j_*} y - \mu\|^2 \le \Cl{lem56} \sigma^2x
        \right\} \ge 1-\Cl{series-probability} e^{-x}.
    \end{equation}
\end{theorem}
Let us explain the ``cost-free'' nature of the above result. In the simplest, one-dimensional regression problem
where the design matrix $X$ has only one column and $\mu=X\beta^*$ for some unknown scalar $\beta^*$,
the prediction error of the Ordinary Least Squares estimator
is $\E[\|X(\hat\beta^{ols} - \beta^*)\|^2] = \sigma^2$ because the random variable 
$\|X(\hat\beta^{ols} -\beta^*)\|^2/\sigma^2$ has chi-square distribution with
one degree-of-freedom.
Hence the right hand side of the regret bound in \eqref{regret-bound-main-theorem-expectation} is no larger than
a constant times the prediction error in a one-dimensional linear model.
The right hand side of \eqref{regret-bound-main-theorem-expectation} is independent of the
minimal risk $R^*$,
independent of the cardinality $M$ of the family of estimators, and if the estimators were constructed
from a linear model with $p$ covariates, the right hand side of \eqref{regret-bound-main-theorem-expectation}
is also independent of the dimension $p$.

Since the most commonly ordered linear smoothers are Tikhonov regularizers (which encompass  Ridge regression and smoothing splines),
we provide the following corollary for convenience.
\begin{corollary}[Application to Tikhonov regularizers]
    Let $K$ be a positive definite matrix of size $p\times p$ and let $\lambda_1,...,\lambda_M\ge 0$
    be distinct tuning parameters. Define $\hat\theta$ as the minimizer of
    \begin{equation}
    \hat\theta = \argmin_{\theta\in\simplex}
    \Big(
    \frac 1 2 \|\sum_{j=1}^M \theta_j X \hat w(K,\lambda_j) - y \|^2
    + 2 \sigma^2 \sum_{j=1}^M\theta_j \df_j
    + \frac 1 2 \sum_{j=1}^M \theta_j
    \|X \hat w(K,\lambda_j) - y\|^2
    \Big),
    \label{objective-function-tikhnonv-convenience}
    \end{equation}
    where $\df_j = \trace[X^T(X^TX + \lambda_j K)^{-1}X^T]$.
    Then the weight vector $\tilde w = \sum_{j=1}^M \hat\theta_j \hat w(K,\lambda_j)$ in $\R^p$ is such that
    the regret \eqref{regret-tikhonov} is bounded from above by $\Cr{punchline}\sigma^2$
    for some absolute constant $\Cr{punchline}>0$.
\end{corollary}
This corollary is a direct consequence of \Cref{thm:main} with $A_j = X^T(X^TX
+ \lambda_j K)^{-1} X^T$.  The fact that this forms a family of ordered linear
smoothers is explained after \eqref{eq:linear-smoothers}.  The objective
function \eqref{objective-function-tikhnonv-convenience} corresponds to
the formulation \eqref{eq:objective-function-clearly-convex} of the objective
function in \eqref{optimization-problem}; we have chosen this formulation
so that \eqref{objective-function-tikhnonv-convenience} can be easily implemented
as a convex quadratic program with linear constraints, the first term of the objective
function being quadratic in $\theta$ while the second and third terms are
linear in $\theta$.

The procedure above requires knowledge of $\sigma^2$, which needs to be estimated beforehand in practice.
Estimators of $\sigma^2$ are available depending on the underlying context,
e.g., difference based estimates for observations on a grid \cite{dette1998estimating,hall1990asymptotically,munk2005difference,brown2007variance},
or pivotal estimators of $\sigma$ in sparse linear regression, e.g.,
\cite{belloni2014pivotal,sun2012scaled,owen2007robust}.
Finally \cite[Section 7.5]{friedman2001elements} recommends estimating $\sigma^2$
by the squared residuals on a low-bias model. We also note that procedure \eqref{optimization-problem} is robust to misspecified $\sigma$ if each $A_j$ is an orthogonal projection
\cite[Section 6.2]{bellec2014affine}.

\section{Multiple families of ordered smoothers or Tikhonov penalty matrices}

\begin{theorem}
\label{thm:multi.family}
    The following holds for absolute constants
    $\Cr{punchline},\Cr{lem56},\Cr{series-probability}>0$.
    Consider the Gaussian mean model \eqref{model}.
    Let $\{A_1,...,A_M\}$ be a set of linear estimators such that
    $$ \{A_1,...,A_M\} \subset F_1\cup ... \cup F_q,
    $$
    where $F_k$ is a family of ordered linear smoothers as in \Cref{def:ordered-linear-smoothers} for each $k=1,...,q$.
    Let $\hat\theta$ be the solution to the optimization problem \eqref{optimization-problem}. Then
    $\hat y = A_{\hat \theta}y$ enjoys the regret bound
    \begin{equation}
    \label{multiple-families-mean}
        \E[ \|A_{\hat\theta}y - \mu\|^2] - \min_{j=1,...,M}\E[\|A_j y - \mu\|^2] \le \Cr{punchline}\sigma^2 + \Cr{lem56} \sigma^2 \log q
        .
    \end{equation}
    Furthermore, if $j_* =\argmin_{j=1,...,M}\E[\|A_j y - \mu\|^2]$ has minimal risk 
    then for any $x\ge 1$,
    \begin{equation}
        \label{multiple-famiilies-probability}
        \mathbb P\left\{
            \|A_{\hat\theta}y - \mu\|^2 - \|A_{j_*} y - \mu\|^2 \le \Cr{lem56} \sigma^2(x+\log q)
        \right\} \ge 1-\Cr{series-probability} e^{-x}.
    \end{equation}
\end{theorem}
We now allow not only one family of ordered linear smoothers, but several.
Above, $q$ denotes the number of families. This setting was considered in
\cite{kneip1994ordered}, although with a regret bound of the form
$\sqrt{R^*}\sigma\log(q)^2 + \sigma^2 \log(q)^4$ where $R^*=\min_{j=1,...,M}\E[\|A_j y - \mu\|^2]$; \Cref{thm:multi.family} improves both the dependence in $R^*$ and in $q$.
Let us also note that the
dependence in $q$ in the above bound \eqref{multiple-famiilies-probability}
is optimal \cite[Proposition 2.1]{bellec2014affine}.

The above result is typically useful in situations where several Tikhonov
penalty matrices $K_1,...,K_q$ are candidate. For each $m=1,...,q$, the penalty matrix
is $K_m$, the practitioner chooses a grid of $b_m \ge 1$ tuning  parameters,
say, $\{\lambda^{(m)}_{a}, a=1,...,b_m \}$.
If the matrices $A_1,...,A_M$ are such that
$$\{A_1,...,A_M\} = \cup_{m=1}^q \{X (X^TX + \lambda_a^{(m)} K_m)^{-1}X^T, a=1,...,b_m \},
$$
so that $M=\sum_{m=1}^q b_m$,
the procedure \eqref{optimization-problem}
enjoys the regret bound
$$\E[\|A_{\hat\theta}y - \mu\|^2]-\min_{m=1,...,q}\min_{a=1,...,b_m}\E[\|X\hat w(K_m,\lambda_a) - \mu\|^2] \le \C\sigma^2(1+\log q)
$$
and a similar bound in probability.
That is, the procedure of \Cref{sec:construction-estimator} automatically
adapts to both the best penalty matrix and the best tuning parameter. The error
term $\sigma^2(1+\log q)$ only depends on the number of regularization
matrices used, not on the cardinality of the grids of tuning parameters.

\section{Proofs}
We start the proof with the following deterministic result.

\begin{lemma}[Deterministic inequality]
    \label{lemma:deterministic}
    Let $A_1,...,A_M$ be square matrices of size $n\times n$ and consider the procedure \eqref{optimization-problem}
    in the unknown mean model \eqref{model}.
    Then for any $\bar A\in\{A_1,...,A_M\}$,
    \begin{equation*}
        \|A_{\hat \theta} y - \mu\|^2
        - 
        \|\bar A y - \mu\|^2
        \le
        \max_{j=1,...,M}
        \Big(
        2\eps^T(A_j - \bar A)y
        - 2\sigma^2\trace(A_j - \bar A)
        - \tfrac 1 2 \|(A_j-\bar A)y\|^2
        \Big).
    \end{equation*}
\end{lemma}
\begin{proof}
    The above is proved in \cite[Proposition 3.2]{bellec2014affine}. We reproduce the short proof here for completeness: If $H:\simplex\to\R$ is the convex objective of \eqref{optimization-problem} and $\bar A = A_k$ for some $k=1,...,M$, the optimality condition
    of \eqref{optimization-problem} states that
    $\nabla H(\hat\theta)(e_k-\hat\theta)\ge0$ holds (cf. \cite[(4.21)]{boyd2009convex}).
    Then $\nabla H(\hat\theta)(e_k-\hat\theta)\ge0$  can be equivalently rewritten as
    $$\|A_{\hat \theta} y - \mu\|^2
        - 
        \|\bar A y - \mu\|^2
        \le
        \sum_{j=1}^M\hat\theta_j
        \Big(
        2\eps^T(A_j - \bar A)y
        - 2\sigma^2\trace(A_j - \bar A)
        - \tfrac 1 2 \|(A_j-\bar A)y\|^2
        \Big).
    $$
    The proof is completed by noting that the average $\sum_{j=1}^M\hat\theta_j
    a_j$ with weights $\hat\theta=(\hat\theta_1,...,\hat\theta_M)\in\simplex$
    is smaller than the maximum $\max_{j=1,...,M}a_j$ for every reals
    $a_1,...,a_M$.
\end{proof}

Throughout the proof, $\bar A$ is a fixed deterministic matrix with $\|\bar A\|_{op}\le 1$.
Our goal is to bound from above the right hand side of \Cref{lemma:deterministic} with high
probability. To this end, define the process $(Z_B)_{B}$ indexed by
symmetric matrices $B$ of size $n\times n$, by
$$Z_B=
    2\eps^T (B-\bar A) y
    - 2\sigma^2\trace(B-\bar A)
    - \tfrac 1 2 (\|(B-\bar A)y\|^2 - d(B,\bar A)^2)
$$
where $d$ is the metric
\begin{equation}
    \label{def-metric}
    d(B,A)^2 \triangleq \E[\|(B-A)y\|^2] = \sigma^2\|B-A\|_F^2 + \|(B-A)\mu\|^2,
    \qquad A,B\in\R^{n\times n}.
\end{equation}
With this definition, the quantity inside the parenthesis in the right hand side
of \Cref{lemma:deterministic} is exactly $Z_{A_j} - \tfrac 1 2 d(A_j,\bar A)$.
We split the process $Z_B$ into a Gaussian part and a quadratic part. Define
the processes $(G_B)_B$ and $(W_B)_B$ by
\begin{align}
    G_B &= \eps^T [2I_{n\times n} - (B-\bar A)/2](B-\bar A)\mu,
          \\
          W_B &= 2\eps^T (B-\bar A)\eps - 2\sigma^2\trace(B-\bar A)
- \tfrac 1 2 \eps^T (B-\bar A)^2 \eps + \tfrac{\sigma^2}{2} \|B-\bar A\|_F^2.
\end{align}
Before bounding supremum of the above processes, we need to derive the following
metric property of ordered linear smoothers.
If $T$ is a subset of the space of symmetric matrices of size $n\times n$
and if $d$ is a metric on $T$, the diameter $\Delta(T,d)$ of $T$ and
the Talagrand generic chaining functionals for each $\alpha=1,2$ are defined by
\begin{equation}
    \label{eq:gamma-functional-definition}
    \Delta(T,d) = \sup_{A,B\in T} d(A,B),
    \qquad\qquad
    \gamma_\alpha(T,d) = \inf_{(T_k)_{k\ge 0}} \sup_{t\in T} \sum_{k=1}^{+\infty} 2^{k/\alpha} d(t, T_k)
\end{equation}
where the infimum is over all sequences $(T_k)_{k\ge 0}$ of subsets of $T$ such that $|T_0|=1$ and
$|T_k| \le 2^{2^k}$.

\begin{lemma}
    \label{lemma:bound-on-gamma-functionals}
    Let $a\ge0$ and let $\mu\in\R^n$.
    Let $F\subset \R^{n\times n}$ be a set of ordered linear smoothers (cf. \Cref{def:ordered-linear-smoothers}) and let $d$ be any semi-metric of the form
    $d(A,B)^2 = a \|A-B\|_F^2 + \|(A-B)\mu\|^2$.
    Then
    $\gamma_2(F,d) + \gamma_1(F,d) \le \Cl{diameter} \Delta(F,d)$ where $\Cr{diameter}$ is an absolute constant.
\end{lemma}
\begin{proof}
    We have to specify a sequence  $(T_k)_{k\ge 0}$ of subsets of $F$ with $|T_k|\le 2^{2^k}$.
    Since $F$ satisfies \Cref{def:ordered-linear-smoothers}, there exists a basis of eigenvectors $u_1,...,u_n$, increasing functions
    $\alpha_1,...,\alpha_n:\R\to[0,1]$ and a set $\Lambda\subset \R$ such that
    $F = \{B_\lambda, \lambda \in \Lambda \}$
    where $B_\lambda = \sum_{i=1}^n\alpha_i(\lambda)u_i u_i^T$,
    cf. \eqref{eq:linear-smoothers}.
    Hence for any $\lambda_0,\lambda,\nu\in\Lambda$,
    \begin{align*}
        d(B_\lambda,B_\nu)^2 &= \sum_{i=1}^n
        w_i(\alpha_i(\lambda)-\alpha_i(\nu))^2
        \quad \text{ for weights }\quad
        w_i = (a+(u_i^T \mu)^2) \ge 0,
        \\
        d(B_{\lambda_0},B_\lambda)^2
        +
        d(B_\lambda, B_\nu)^2
        &= 
        d(B_{\lambda_0},B_\nu)^2
        +
        2\sum_{i=1}^n w_i
        (\alpha_i(\lambda)-\alpha_i(\lambda_0))
        (\alpha_i(\lambda)-\alpha_i(\nu))
        .
    \end{align*}
    If $\lambda_0\le\lambda\le\nu$, since each $\alpha_i(\cdot)$ is nondecreasing,
    the sum in the right hand side of the previous
    display is non-positive and $d(B_{\lambda},B_\nu)^2 \le d(B_{\nu},B_{\lambda_0})^2 - d(B_\lambda,B_{\lambda_0})^2$ holds.
    Let $N=2^{2^k}$ and $\delta = \Delta(F,d)/N$. We construct a $\delta$-covering
    of $F$ by considering the bins 
    $\bin_j=\{B\in F: \delta^2 j \le d(B,B_{\lambda_0})^2 < \delta^2 (j+1)\}$ for $j=0,...,N-1$ where $\lambda_0=\inf \Lambda$.
    If $\bin_j$ is non-empty, any of its element is a $\delta$-covering of $\bin_j$ thanks
    to
    $$d(B_\lambda,B_\nu)^2
    \le d(B_\nu,B_{\lambda_0})^2 - d(B_\lambda,B_{\lambda_0})^2
    \le (j+1)\delta^2 - j\delta^2 = \delta^2.
    $$
    for $B_\nu,B_\lambda\in\bin_j$ with $\lambda\le\nu$.
    This constructs a $\delta$-covering of $F$ with $N=2^{2^k}$ elements.
    Hence $\gamma_2(F,d)\le\Delta(F,d) \sum_{k=1}^\infty 2^{k/2}/ 2^{2^k}  = \Delta(F,d) \C$ and the same holds for $\gamma_1(F,d)$ for a different absolute constant.
\end{proof}

\begin{lemma}[The Gaussian process $G_B$]
    \label{lemma:gaussian-process}
    Let $T^*$ be a family of ordered smoothers (cf. \Cref{def:ordered-linear-smoothers})
    such that $\sup_{B\in T^*}d(\bar A,B)\le \delta^*$ for the metric \eqref{def-metric}.
    Then for all $x>0$,
    $$
    \mathbb P(\sup_{B\in T^*}G_B \le \sigma(\C +3\sqrt{2x})\delta^* ) \ge 1 - e^{-x}.
    $$
\end{lemma}
\begin{proof}
    By the Gaussian concentration theorem \cite[Theorem 5.8]{boucheron2013concentration},
with probability at least $1-e^{-x}$ we have
\begin{align}
    \sup_{B\in T^*} G_B &\le
\E\sup_{B\in T^*} G_B + \sigma\sqrt{2x} \sup_{B\in T^*}\|[2I_{n\times n} - (B-\bar A)/2](B-\bar A)\mu\|.
\\
&\le
\C \gamma_2(T^*,d_G) + \sigma\sqrt{2x} \sup_{B\in T^*} 3 \|(B-\bar A)\mu\|
\end{align}
where for the second inequality we used Talagrand's majorizing measure theorem (cf., e.g.,
\cite[Section 8.6]{vershynin2018high}) and the fact that $B,\bar A$ have operator norm at most one, where $d_G$ is the canonical metric of the Gaussian process,
$$d_G(A,B)^2 = \E[(G_A-G_B)^2].
$$
If $D=B-A$ is the difference and $P$ commute with $A$ and $B$,
$$G_B-G_A =
    \epsilon^T\left[2D\mu-\tfrac 1 2 (A+B-2\bar A)D\mu-\tfrac 1 2 D(A+B-2P)\mu\right]
     +\epsilon^T D\left(\bar{A}-P\right)\mu.
$$
By the triangle inequality and using that $A,B,P,\bar A$ have operator norm at most one,
$d_G(A,B) \le 6\sigma\|D\mu\| + \sigma\|D(\bar A-P)\mu\|.
$
This shows that 
$$\gamma_2(T^*,d_G)\le 6\sigma \gamma_2(T^*,d_1) + \sigma \gamma_2(T^*,d_2)$$
where $d_1(A,B)=\|(B-A)\mu\|$ and $d_2(A,B)=\|(A-B)(\bar A - P)\mu\|$.
By \Cref{lemma:bound-on-gamma-functionals},
$\gamma_2(T^*,d_1)\le \C \Delta(T^*,d_1)$ and similarly for $d_2$
(note that $d_2$ is similar to $d_1$ with $\mu$ replaced by $\mu'=(P-\bar A)\mu$).

If $\sup_{B\in T^*}d(B,\bar A)\le \delta^*$ for the metric $d$ in \eqref{def-metric},
then $\sup_{B\in T^*}\|(B-\bar A)\mu\|\le \delta^*$ 
and $\Delta(T^*,d_1)\le 2 \delta^*$.
Furthermore if $P$ is the convex projection of $\bar A$ onto the convex hull of $T^*$ with
respect to the Hilbert metric $d$ in \eqref{def-metric}, then
$$
\Delta(T^*,d_2)=\sup_{B,B'\in T^*}
d_2(B,B')
\le
2
\|(P-\bar A)\mu\|
\le
2 d(P,\bar A)
\le
2 d(B_0,\bar A)
\le 2\delta^*
$$
for any $B_0\in T^*$ where we used that by definition of the convex projection,
$d(P,\bar A)\le d(B_0,\bar A)$.
\end{proof}
The following inequality, known as the Hanson-Wright inequality, will be useful
for the next Lemma.
If $\eps\sim N(0,\sigma^2 I_{n\times n})$ is standard normal, then
\begin{equation}
    \label{hanson-wright}
    \mathbb P\Big[|\eps^T Q \eps - \sigma^2 \trace Q |> 2\sigma^2(\|Q\|_F\sqrt{x} + \|Q\|_{op}x )
    \Big]\le 2 e^{-x},
\end{equation}
for any square matrix $Q\in\R^{n\times n}$.
We refer to \cite[Example 2.12]{boucheron2013concentration} for a proof for normally
distributed $\eps$ and
\cite{rudelson2013hanson,hsu2012tail,bellec2014bernstein,adamczak2015note} for
proofs of \eqref{hanson-wright} in the sub-gaussian case. 

\begin{lemma}[The Quadratic process $W_B$]
    \label{lemma:quadratic-process}
    Let $T^*$ be a family of ordered smoothers (cf. \Cref{def:ordered-linear-smoothers})
    such that $\sigma\|B-\bar A\|_F\le \delta^*$ for all $B\in T^*$. Then for all $x>0$,
    $$
    \mathbb P\Big(
        \sup_{B\in T^*} W_B \le \C \sigma\delta^* + \C\sigma\sqrt{x}\delta^* + \C \sigma^2x 
    \Big) \ge 1 - 2e^{-x}.
    $$
\end{lemma}
\begin{proof}
    We apply Theorem 2.4 in \cite{adamczak2015note} which implies that
    if $W_B=\eps^TQ_B\eps - \trace[Q_B]$ where $\eps\sim N(0,I_{n\times n})$
    and $Q_B$ is a symmetric matrix of size $n\times n$ for every $B$, then
    $$
    \mathbb P\Big(
    \sup_{B\in T^*} W_B
    \le
    \E
    \sup_{B\in T^*} W_B
    + \C \sigma \sqrt x \sup_{B\in T^*}
    \E\|Q_B \eps\|
    + \C x \sigma^2 \sup_{B\in T^*} \|Q_B\|_{op}
    \Big) \ge 1 - 2e^{-x}.
    $$
    For the third term, $Q_B=2(B-\bar A) - (B-\bar A)^2/2$ hence
    $\|Q_B\|_{op}\le 6$ because $B,\bar A$ both have operator norm at most one.
    For the second term, since $T^*$ is a family of ordered linear smoothers,
    there exists extremal matrices $B_0,B_1\in T^*$ such
    that $B_0\preceq B \preceq B_1$ for all $B\in T^*$;
    we then have $B-B_0 \preceq B_1-B_0$ and
    $$
        \|Q_B\eps\|
        \le 3\|(B-\bar A)\eps\|
        \le 3 \|(B_1-B_0)\eps\| + 3 \|(B_0-\bar A)\eps\|
        \le 3 \|(B_1-\bar A)\eps\| + 6 \|(B_0-\bar A)\eps\|.
    $$
    Hence  $\E\|Q_B\eps\| \le \E[\|Q_B\eps\|^2]^{1/2}\le 3\sigma \|B_1-\bar A\|_F + 6\sigma \|B_0-\bar A\|_F \le 9 \delta^*$.

    We finally apply a generic chaining upper bound to bound $\E\sup_{B\in T^*} W_B$.
    For any fixed $B_0\in T^*$ we have $\E[W_{B_0}]=0$ hence
    $\E\sup_{B\in T^*} W_B =\E\sup_{B\in T^*}(W_B-W_{B_0})$.
    For two matrices $A,B\in T^*$ we have
    $W_B-W_A = \eps^T(Q_B-Q_A)\eps-\trace[Q_B-Q_A]$,
    and
    $$\eps^T(Q_B-Q_A)\epsilon = \eps^T[(B-A)(2I_{n\times n} - \tfrac 1 2 (A+B-2\bar A) ) ] \eps,$$ 
    hence by the Hanson-Wright inequality \eqref{hanson-wright}, with probability at least
    $1-2e^{-x}$,
    $$
        |W_B-W_A|
        \le 2\sigma^2\|
            (B-A)(2 I_{n\times n} - \tfrac 1 2 (A+B-2\bar A) )
        \|_F(\sqrt x + x)
        \le 8 \sigma^2 \|A-B\|_F(x+\sqrt x).
    $$
    Hence by the generic chaining bound given in Theorem 3.5 in \cite{dirksen2015tail},
    we get that
    $$\E\sup_{B\in T^*}|W_B-W_{B_0}| \le \C\sigma^2 \left[ \gamma_1(T^*,\|\cdot\|_F) + \gamma_2(T^*,\|\cdot\|_F) + \Delta(T^*,\|\cdot\|_F) \right].
    $$
    For each $\alpha=1,2$ we have
    $\gamma_\alpha(T^*,\|\cdot\|_F)\le \C \Delta(T^*,\|\cdot\|_F)$ by \Cref{lemma:bound-on-gamma-functionals}.
    Since $\sigma\|B-\bar A\|\le \delta^*$ for any $B\in T^*$,
    we obtain $\Delta(T^*,\|\cdot\|_F)\le 2 \delta^*/\sigma$.
\end{proof}

\begin{lemma}
    \label{lemma:one-slice}
    Suppose $F$ is a family of $n\times n$ ordered linear smoothers (cf.
    \Cref{def:ordered-linear-smoothers}), and $\bar A$ is a fixed matrix with $\|\bar A\|_{op}\leq 1$ which may not belong to $F$.
    Let $d$ be the metric \eqref{def-metric}. 
    Then for any reals $u\ge 1,$ and $\delta^* > \delta_*\ge 0$, we have with probability at least $1-3 e^{-u}$,
    $$
    \sup_{B\in F:\; \delta_* \le d(B,\bar A) < \delta^*}
    \left(Z_B - \tfrac 1 2 d(B,\bar A)^2 \right)
    \le \Cl{lem54-upper-1}\left[\sigma^2 u + \delta^* \sigma \sqrt{u}\right] - \tfrac 1 2 \delta_*^2
    \le \Cl{lem54-upper-2}\sigma^2 u + \tfrac{1}{16}(\delta^*)^2 - \tfrac 1 2 \delta_*^2.
    $$
\end{lemma}
\begin{proof}
    First note that $-d(B,\bar A)^2\le-\delta_*^2$ for any $B$ as in the supremum.

    Now $Z_B=G_B+W_B$ where $G_B$ and $W_B$ are the processes studied
    in \Cref{lemma:gaussian-process,lemma:quadratic-process}.
    These lemmas applied to $T^*= \{ B\in F: d(B,\bar A)\le \delta^* \}$ yields
    that on an event of probability at least $1-3e^{-u}$ we have
    $$
    \textstyle
    \sup_{B\in T^*}Z_B \le \sup_{B\in T^*}(G_B + W_B)
    \le \C(\sigma \delta^* (1+\sqrt u)  + \sigma^2 u).
    $$
Since $u\ge 1$, we have established the first inequality by adjusting the absolute constant.
For the second inequality, we use that
$\Cr{lem54-upper-1}\delta_*\sigma\sqrt u \le 4\Cr{lem54-upper-1}^2\sigma^2 u + \tfrac{1}{16} (\delta^*)^2
$
and set $\Cr{lem54-upper-2}=\Cr{lem54-upper-1} + 4\Cr{lem54-upper-1}^2$.
\end{proof}

\begin{lemma}[Slicing]
    \label{lemma:final-slicing}
    Suppose $F$ is a family of $n\times n$ ordered linear smoothers (cf.
    \Cref{def:ordered-linear-smoothers}), and $\bar A$ is a fixed matrix with $\|\bar A\|_{op}\leq 1$ which may not belong to $F$.
    Let $d$ be the metric \eqref{def-metric}. 
    Then for any $x\ge 1$, we have with probability at least $1- \Cr{series-probability} e^{-x}$
    $$
    \textstyle
    \sup_{B\in F}
    \left(Z_B - \tfrac 1 2 d(B,\bar A)^2 \right)
    \le \Cr{lem56} \sigma^2 x.
    $$
\end{lemma}

\begin{proof}
    We use here a method known as \emph{slicing}, we refer the reader to
    Section 5.4 in~\cite{van2014probability} for an introduction. Write $F$ as the union
\[
    F=\cup_{k=1}^\infty T_k 
    \quad \text{ where } T_k\text{ is the slice }\quad
    T_k=\{B\in F: \delta_{k-1}\leq \tilde{d}(B,\bar A)\leq \delta_k\},
\]
with $\delta_0=0$ and $\delta_k=2^k\sigma$ for $k\geq 1$.
By definition of the geometric sequence $(\delta_k)_{k\ge0}$, inequality
$\frac{1}{16}\delta_k^2-\frac 1 2 \delta_{k-1}^2 \le \frac 1 2 \sigma^2
- \frac{1}{16}\delta_k^2
\le \frac 1 2 \sigma^2 x
- \frac{1}{16}\delta_k^2
$ holds for all $k\ge 1$.
With $\delta_*=\delta_{k-1},\delta^*=\delta_k$, \Cref{lemma:one-slice} yields
that for all $k\ge1$,
$$
\mathbb P\Big(
    \sup_{B\in T_k}(Z_B - \tfrac 1 2 d(B,\bar A)^2) \le \Cr{lem54-upper-2}\sigma^2 u_k - \tfrac{1}{16} \delta_k^2 + \tfrac{\sigma^2 x}{2}
\Big) \ge 1 - 3e^{-u_k}
$$
for any $u_k\ge 1$.
The above holds simultaneously over all slices $(T_k)_{k\ge 1}$
with probability at least $1-3\sum_{k=1}^\infty e^{-u_k}$ by the union bound.
It remains to specify a sequence $(u_k)_{k\ge 1}$ of reals greater than 1.
We choose $u_k = x+\delta_k^2/(\sigma^2 16\Cr{lem54-upper-2})$ which is greater than 1 since $x\ge 1$.
Then by construction,
$\Cr{lem54-upper-2}\sigma^2 u_k - \tfrac{1}{16} \delta_k^2 + \tfrac{\sigma^2 x}{2}
= (\Cr{lem54-upper-2}+1/2)\sigma^2 x$ and we set
$\Cr{lem56}=\Cr{lem54-upper-2}+1/2$.
Furthermore,
$\sum_{k=1}^\infty e^{-u_k} = e^{-x} \sum_{k=1}^\infty e^{-2^{2k}/(16\Cr{lem54-upper-2})}.$ 
The sum $3\sum_{k=1}^\infty e^{-2^{2k}/(16\Cr{lem54-upper-2})}$ is equal to a finite absolute constant
named $\Cr{series-probability}$ in the statement of the Lemma.
\end{proof}

\begin{proof}[Proof of \Cref{thm:main}]
    Let $F=\{A_1,...,A_M\}$ and $\bar A = A_{j_*}$ where $j_*$ is defined in
    the statement of \Cref{thm:main}. The conclusion of \Cref{lemma:deterministic}
    can be rewritten as
    $$\|A_{\hat\theta} y - \mu\|^2 - \|\bar A y - \mu\|^2 \le \sup_{B\in F}(Z_B - \tfrac 1 2 d(B,\bar A)^2)
    $$
    where $F=\{A_1,...,A_M\}$ is a family of ordered linear smoothers.
    \Cref{lemma:final-slicing} completes the proof of
    \eqref{eq:main-thm:probability}.
    Then \eqref{regret-bound-main-theorem-expectation} is obtained by integration
    of \eqref{eq:main-thm:probability} using $\E[Z]\le \int_0^\infty \mathbb P(Z>t) dt$
    for any $Z\ge 0$.
\end{proof}

\begin{proof}[Proof of \Cref{thm:multi.family}]
As in the proof of \Cref{thm:main}, we use \Cref{lemma:deterministic} to deduce that
a.s.,
$$
\|A_{\hat\theta} y - \mu\|^2 - \|\bar A y - \mu\|^2 \leq \max_{j=1,...,M}(Z_{A_j}-\tfrac 1 2 d(A_j,\bar A)^2)= \max_{k=1,...,q}\max_{B\in F_k} (Z_B -\tfrac 1 2 d(B,\bar A)^2).
$$
Since each $F_k$ is a family of ordered linear smoothers, by \Cref{lemma:final-slicing} we have
$$
\textstyle
\mathbb{P}\big(\max_{B\in F_k} (Z_B -\tfrac 1 2 d(B,\bar A)^2)>\Cr{lem56}\sigma^2 x\big)\leq \Cr{series-probability}e^{-x}
\qquad
\text{ for each } k=1,\ldots, q.
$$
The union bound yields \eqref{multiple-famiilies-probability} and we use $\E[Z]\le \int_0^\infty \mathbb P(Z>t) dt$ for $Z\ge0$ to deduce \eqref{multiple-families-mean}.
\end{proof}

\bibliographystyle{plainnat}
\bibliography{nipsBY}

\end{document}